\documentclass[a4paper, 11pt]{article}

\usepackage{enumerate}
\usepackage{authblk}
\usepackage{comment}
\usepackage{amssymb}
\usepackage{amsthm}
\usepackage{amscd}
\usepackage{graphicx}
\usepackage{titlesec}
\usepackage{indentfirst}
\titleformat{\section}{\bfseries}{\thesection .}{0.5em}{}
\titleformat{\subsection}{\itshape}{\thesubsection .}{0.5em}{}
\usepackage[top=25.4mm, bottom=25.4mm, left=27.7mm, right=27.8mm]{geometry}

\usepackage[tbtags]{amsmath}
\usepackage{booktabs}
\allowdisplaybreaks  

\usepackage{caption}
\begin{document}

\newtheorem{theorem}{Theorem}[section]
\newtheorem{corollary}[theorem]{Corollary}
\newtheorem{lemma}[theorem]{Lemma}
\newtheorem{proposition}[theorem]{Proposition}
\theoremstyle{definition}
\newtheorem{definition}[theorem]{Definition}
\newtheorem{example}[theorem]{Example}
\newtheorem{remark}[theorem]{Remark}

\title{\Large $L^0$--convex compactness and  random normal structure in $L^0(\mathcal{F},B)$  }
\author{
{\normalsize Tiexin Guo\thanks{Corresponding author\\
\textit{Email addresses}: tiexinguo@csu.edu.cn (Tiexin Guo), zhangerxin6666@163.com (Erxin Zhang), wychao@csu.edu.cn (Yachao Wang).}, Erxin Zhang, Yachao Wang}\\

\textit{\small School of Mathematics and Statistics, Central South University, Changsha 410083, China.}\\

{\normalsize George Yuan \thanks{\\
\textit{Email addresses}: george\_yuan@yahoo.com (George Yuan)}}\\

\textit{\small School of Mathematics, Shanghai University of Finance and Economics, Shanghai 200433, China.}}

\date{}
\maketitle

\renewcommand{\baselinestretch}{1.2}
\large\normalsize
\noindent \rule[0.5pt]{14.5cm}{0.6pt}\\
\noindent
\textbf{Abstract}\\
Let $(B,\|\cdot\|)$ be a Banach space, $(\Omega,\mathcal{F},P)$ a probability space and $L^0(\mathcal{F},B)$ the set of equivalence classes of strong random elements (or strongly measurable functions) from $(\Omega,\mathcal{F},P)$ to $(B,\|\cdot\|)$. It is well known that $L^0(\mathcal{F},B)$ becomes a complete random normed module, which has played an important role in the process of applications of random normed modules to the theory of Lebesgue--Bochner function spaces and random functional analysis. Let $V$ be a closed convex subset of $B$ and $L^0(\mathcal{F},V)$ the set of equivalence classes of strong random elements from $(\Omega,\mathcal{F},P)$ to $(B,\|\cdot\|)$, the central purpose of this paper is to prove the following two results: (1). $L^0(\mathcal{F},V)$ is $L^0$--convexly compact if and only if $V$ is weakly compact; (2). $L^0(\mathcal{F},V)$ has random normal structure if $V$ is weakly compact and has normal structure. As an application, a general random fixed point theorem for a strong random nonexpansive operator is given, which generalizes and improves several well known results. We hope that our new method, namely skillfully combining measurable selection theorems, the theory of random normed modules and Banach space techniques, can be applied in the other related aspects.

\vspace{0.3cm}

\noindent \rule[0.5pt]{14.5cm}{0.6pt}
\section{Introduction}  
\par
Random normed modules (briefly, $RN$ modules) are a random generalization of ordinary normed spaces. The theory of $RN$ modules has undergone a systematic and deep development \cite{Guo89,Guo92,Guo93,Guo95,Guo96a,Guo08,Guo10,Guo13,GL05,GY96,GZ10,GZ12} and has also been applied to the study of conditional risk measures \cite{GZWW17,GZWYYZ17,GZZ14,GZZ15a,GZZ15b} and backward stochastic equations \cite{GZWG18}. Motivated by financial applications, we recently presented the notion of $L^0$--convex compactness in \cite{GZWW17} and further the notion of random normal structure for an $L^0$--convex subset of an $RN$ module in \cite{GZWG18}, in particular we proved in \cite{GZWG18} $($ see also Section 2 of this paper for details$)$ the Browder--Kirk's fixed point theorem in a complete $RN$ module, which is a random generalization of the famous Browder--Kirk's fixed point theorem in a Banach space \cite{B65,GR84,G65,K65,K83}. The generalized Browder--Kirk's fixed point theorem has been applied to the study of backward stochastic equations and backward stochastic differential equations of nonexpansive type in \cite{GZWG18}. In this paper we find that the fixed point theorem can be further used to study Bharucha--Reid's problem, which we introduce in the following statement of backgrounds of random fixed point theory. Let $(B,\|\cdot\|)$ be a Banach space, $(\Omega,\mathcal{F},P)$ a probability space, $V\subset B$ a closed convex subset and $L^0(\mathcal{F},B)(L^0(\mathcal{F},V))$ the set of equivalence classes of $B$--valued$($ correspondingly, $V$--valued$)$ strong $\mathcal{F}$--random elements defined on $\Omega$, then $L^0(\mathcal{F},B)$ becomes an $RN$ module in a natural fashion and $L^0(\mathcal{F},V)$ a closed $L^0$--convex subset of $L^0(\mathcal{F},B)($ see Section 2 of this paper$)$. In the process of applying the generalized Browder--Kirk's fixed point theorem to Bharucha--Reid's problem, it is very key to solve the following two problems:
\begin{enumerate}[(1)]
\item When $V$ is weakly compact, is $L^0(\mathcal{F},V)$ $L^0$--convexly compact ?
\item When $V$ is a weakly compact convex subset with normal structure, does $L^0(\mathcal{F},V)$ have random normal structure ?
\end{enumerate}
\par
Random fixed point theory for random operators was initiated by A.$\check{S}$pac$\check{e}$k and O.Han$\check{s}$ in \cite{S55,H57a,H57b,H61} for the study of random operator equations, The main difficulty lies in establishing measurability of random fixed points. For the fixed points of contractive random operators, their measurability can be established by successive approximant methods \cite{H57b}. But for the random generalization of topological fixed point theorems like Schauder's fixed point theorem, the measurability problem can be solved by measurable selection theorems as surveyed in \cite{W77}, see, e.g.,\cite{B76}. The two kinds of methods both heavily depend on the separability of spaces or sets in question. After Bharucha--Reid \cite{B76} surveyed the random generalizations of Banach's and Schauder's fixed point theorems, he also ever presented the problem of random generalization of fixed point theorems for nonexpansive mappings$($see $(i)$ in Section 5 of \cite{B76}$)$. Browder--Kirk's fixed point theorem\cite{B65,GR84,G65,K65,K83} is, without doubt, the most famous one for nonexpansive mappings, which can be stated as follows: let $(B,\|\cdot\|)$ be a Banach space, $V\subset B$ a weakly compact convex subset with normal structure and $T: V\rightarrow V$ a nonexpansive mapping, then $T$ has a fixed point in $V$. Let $(\Omega,\mathcal{F},P)$ be a probability space, $(M,d)$ and $(M_1.d_1)$ two metric spaces and $T:\Omega\times M\rightarrow M_1$ a mapping, let us recall from \cite{B72,B76,Wang62} : $(1)$ a mapping $X:\Omega\rightarrow M_1$ is said to be a random element$($or, $\mathcal{F}$--random element$)$ if $X^{-1}(G):=\{\omega\in \Omega: X(\omega)\in G\}\in \mathcal{F}$ for each open set $G$ of $M_1$, furthermore, a random element $X$ is said to be simple if it only takes finitely many values, and $X$ is said to be a strong random element if $X$ is the pointwise limit of a sequence of simple random elements, it is well known that $X:\Omega \rightarrow M_1$ is a strong random element if and only if both $X$ is a random element and its range $X(\Omega)$ is a separable subset of $M_1$; $(2)$ $T$ is called a random operator if $T(\cdot,x):\Omega\rightarrow M_1$ is a random element for each $x\in M$, furthermore $T$ is called a strong random operator if $T(\cdot,x)$ is a strong random element for each $x\in M$. Then Bharucha--Reid's problem can be precisely stated as follows: let $(B,\|\cdot\|)$ be a Banach space, $V\subset B$ a weakly compact convex subset with normal structure, $(\Omega,\mathcal{F},P)$ a probability space and $T:\Omega\times V\rightarrow V$ a random nonexpansive operator$($namely, for each $\omega\in \Omega, \|T(\omega,v_1)-T(\omega,v_2)\|\leq\|v_1-v_2\|$ for all $v_1,v_2\in V)$, then, is there a random element $v:\Omega\rightarrow V$ such that $T(\omega,v(\omega))=v(\omega)$ for almost all $\omega\in \Omega$ ?
\par
Lin\cite{L88} and Xu\cite{Xu90} first studied Bharucha--Reid's problem, in particular, Xu\cite{Xu90} partly solved the problem under the assumption that $V$ is a nonempty closed bounded convex separable subset of a reflexive Banach space $B$ such that $V$ has the fixed point property for nonexpansive mappings. It is obvious that $T$ must be a strong random operator when $B$ or $V$ is separable. In this paper we prove that $T$ still has a random fixed point under the weaker assumption that $T$ is only required to be a strong random operator without the requirement that $V$ or $B$ is separable. Our result has an advantage: it includes the classical Browder--Kirk's fixed point theorem as a special case and thus is also more natural. As compared with Lin\cite{L88} and Xu\cite{Xu90}, their methods are making use of measurable selection theorems, whereas ours are based on the recent development of the theory of random normed modules$($ briefly, $RN$ modules$)$ since measurable selection theorems can no longer apply to the case when $B$ or $V$ is not separable. Although we also apply measurable selection theorems, measurable selection theorems are used in this paper in order to deeply develop the theory of $RN$ modules.
\par
The success of this paper lies in answering the above--mentioned problems in a positive way, in particular our method is a skillful combination of the theory of $RN$ modules, measurable selection theorems and Banach space techniques. We may hope that this method can be applied in other related aspects.
\par
The remainder of this paper is organized as follows: in Section 2 of this paper we give some necessary preliminaries on $RN$ modules and further answer the problems $(1)$ and $(2)$ as above; in Section 3 we prove a general random fixed point theorem for a strong random nonexpansive operator on a Banach space as a better solution to Bharucha--Reid's problem.
\section{Some necessary preliminaries on $RN$ modules and positive answers to the problems $(1)$ and $(2)$}
The main results of this section are Theorems 2.11 and 2.16, let us first give some preliminaries.
\par
Throughout this paper, unless otherwise stated, $(\Omega,\mathcal{F},P)$ always denotes a given probability space; $K$ the scalar field $R$ of real numbers or $C$ of complex numbers; $L^0(\mathcal{F},K)$ the algebra of equivalence classes of $K$--valued $\mathcal{F}$--measurable random variables on $(\Omega,\mathcal{F},P)$, where the scalar multiplication, addition and multiplication operations on equivalence classes are, as usual, induced from the corresponding pointwise operations on random variables; $L^0(\mathcal{F})$ simply denotes $L^0(\mathcal{F},R)$ and $\bar{L}^0(\mathcal{F})$ the set of equivalence classes of extended real--valued $\mathcal{F}$--measurable random variables on $(\Omega,\mathcal{F},P)$.
\par
Proposition \ref{proposition2.1} below can be naturally regarded as a randomized version of the well known supremum or infimum principle for $\bar{R}:=[-\infty, +\infty]$ and $R$.

\begin{proposition}\cite{DS58}\label{proposition2.1}
When $\bar{L}^0(\mathcal{F})$ is partially ordered by $\xi \leq \eta$ if and only if $($briefly, iff$)$ $\xi^{0}(\omega) \leq \eta^{0}(\omega)$ for $P$--almost all $\omega$ in $\Omega$, where $\xi^{0}$ and $\eta^{0}$ are respectively arbitrarily chosen representatives of $\xi$ and $\eta$ in $\bar{L}^0(\mathcal{F})$, $(\bar{L}^{0}(\mathcal{F}),\leq)$ is a complete lattice. As usual, $\bigvee H$ and $\bigwedge H$ respectively stand for the supremum and infimum of a subset $H$ of $\bar{L}^0(\mathcal{F})$. Furthermore, the following hold:
\begin{enumerate}[(1)]
\item For every subset $H$ of $\bar{L}^0(\mathcal{F})$, there exist two sequences $\{ a_n, n \in N \}$ and $\{ b_n, n \in N\}$ in $H$ such that $\bigvee H = \bigvee_{n \geq 1} a_n $ and $\bigwedge H = \bigwedge_{n \geq 1} b_n $.
\item If $H$ is directed upwards ,namely there exists some $h_3\in H$ for any given two elements $h_1$ and $h_2$ in $H$  such that $h_1 \bigvee h_2 \leq  h_3$, then $\{a_n, n \in N\}$  in $(1)$ can be chosen as nondecreasing. Similarly, if $H$ is directed downwards, then $\{b_n, n\in N\}$ in $(1)$ can be chosen as nonincreasing.
\item $L^0(\mathcal{F})$, as a sublattice of $\bar{L}^0(\mathcal{F})$, is Dedekind complete, namely any subset having an upper bound in $L^0(\mathcal{F})$ must possess a supremum in $L^0(\mathcal{F})$.
\end{enumerate}
\end{proposition}

As usual, for two elements $\xi$ and $\eta$ in $\bar{L}^0(\mathcal{F}), \xi < \eta$ means $\xi \leq \eta$ and $\xi \neq \eta$. Very often, for any $A\in \mathcal{F}, \xi < \eta$ on $A$ means $\xi^0(\omega) <\eta^0(\omega)$ for $P$--almost all $\omega$ in $A$, where  $\xi^0$ and $\eta^0$  are arbitrarily chosen representatives of $\xi$ and $\eta$, respectively. Let $A = \{ \omega \in \Omega ~:~ \xi^0(\omega) < \eta^0(\omega) \}$, we also use $(\xi <\eta)$ for $A$, which will not produce any confusion, since such sets $A$ only differ by a null set for respectively different choices of $\xi^0$ and $\eta^0$. Similarly, one can understand $(\xi \leq \eta)$, $(\xi \neq \eta)$, $(\xi =\eta)$ and so on.
\par
This paper always employs the following notations:
\par
$L^0_+(\mathcal{F}) := \{ \xi \in L^0(\mathcal{F}) ~:~ \xi \geq 0 \}$;
\par
$\bar{L}^0_+(\mathcal{F}) := \{ \xi \in \bar{L}^0(\mathcal{F}) ~:~ \xi \geq 0 \}$;
\par
$L^0_{++}(\mathcal{F}) := \{ \xi \in L^0(\mathcal{F}) ~:~ \xi >0$ on $\Omega \}$;
\par
$\bar{L}^0_{++}(\mathcal{F}) := \{ \xi \in \bar{L}^0(\mathcal{F}) ~:~ \xi >0$ on $\Omega \}$.

\par
In the study of applying the theory of $RN$ modules to the theory of random operators, one often has to distinguish a random variable from its equivalence class. Therefore, this paper also mentions an equivalent variant of Proposition \ref{proposition2.1} as follows:

\textbf{Proposition2.1$'$}\cite{HWY92}\label{Proposition2.1'}
Let $\bar{\mathcal{L}}(\mathcal{F})$ be the set of extended real--valued $\mathcal{F}$--random variables on $(\Omega,\mathcal{F},P)$. $\eta\in \bar{\mathcal{L}}(\mathcal{F})$ is called an essential upper bound for a subset $H$ of $\bar{\mathcal{L}}(\mathcal{F})$ if $h(\omega)\leq \eta(\omega)$ for $P$--almost all $\omega$ in $\Omega($ briefly, $h\leq \eta$ a.s.$)$ for any $h\in H$. Further, $\eta$ is called an essential supremun of $H$ if $\eta\leq \eta'$ a.s. for any essential upper bound $\eta'$ for $H$. Similarly, one has the notion of an essential lower bound or an essential infimum. Then every subset $H$ of $\bar{\mathcal{L}}(\mathcal{F})$ has an a.s. unique essential supremum and infimum, denoted by ess.sup$(H)$ and ess.inf$(H)$ respectively. Furthermore, $\bigvee H$ and $\bigwedge H$ have the properties $(1)$ to $(3)$ similar to Proposition \ref{proposition2.1}.

\begin{remark}\label{remark2.2}
Let $\mathcal{L}(\mathcal{F})$ be the set of real--valued $\mathcal{F}$--random variables on $(\Omega,\mathcal{F},P)$ and $H$ a subset of $\mathcal{L}(\mathcal{F})$. For a real--valued function $r:\Omega \rightarrow R$, if $h\leq r$ a.s. for any $h\in H$, then it is easy to see ess.sup$(H)\leq r$ a.s., so ess.sup$(H)$ can be taken as an element of $\mathcal{L}(\mathcal{F})$, where $h\leq r$ a.s. means that there is an $\mathcal{F}$--measurable $\Omega'$ with $P(\Omega')=1$ such that $h(\omega)\leq r(\omega)$ for any $\omega\in \Omega'$.
\end{remark}

\begin{definition}\cite{Guo92,Guo93,Guo10}\label{definition2.3}
An ordered pair $(E,\|\cdot\|)$ is called a random normed module $($briefly, an $RN$ module$)$ over $K$ with base $(\Omega, \mathcal{F},P)$ if $E$ is a left module over the algebra $L^0(\mathcal{F},K)$ and $\|\cdot\|$ is a mapping from $E$ to $L^0_+(\mathcal{F})$ such that the following hold:
\begin{enumerate}[(RN-1)]
\item $\|\xi  x \| = |\xi| \cdot \|x\|$ for any $\xi \in L^0(\mathcal{F},K) $ and $x\in E$, where $\xi x$ denotes the module multiplication of $\xi$ and $x$;
\item $\|x + y\| \leq \|x\| + \|y\|$ for any $x$ and $y \in E$;
\item $\|x\| = 0$ implies $x = \theta$ $($the null in $E$$)$.
\end{enumerate}
Here, $\|\cdot\|$ is often called the $L^0$--norm on $E$. If $\|\cdot\|$ only satisfies $(RN-1)$ and $(RN-2)$, then it is called an $L^0$--seminorm on $E$.
\end{definition}
\par
When $(\Omega, \mathcal{F},P)$ is trivial, namely $\mathcal{F}=\{\Omega,\emptyset\}, (E,\|\cdot\|)$ reduces to an ordinary normed space over $K$. Just as a norm induces the norm topology on a normed space, the $L^0$--norm on an $RN$ module $(E,\|\cdot\|)$ induces a metrizable linear topology on $E$, called the $(\varepsilon,\lambda)$--topology. The $(\varepsilon,\lambda)$--topology has its origin in the theory of probabilistic metric spaces, see \cite{SS83} for more historical backgrounds.

\begin{proposition}\cite{Guo92,Guo93,Guo10}\label{proposition2.4}
Let $(E,\|\cdot\|)$ be an $RN$ module over $K$ with base $(\Omega,\mathcal{F},P)$. For any given positive numbers $\varepsilon$ and $\lambda$  with $0 < \lambda <1$, let $N_{\theta}(\varepsilon,\lambda) = \{ x \in E ~:~ P\{\omega \in \Omega ~:~ \|x\|(\omega) < \varepsilon \} > 1-\lambda \}$, called the $(\varepsilon,\lambda)$--neighborhood of $\theta$, then $\{ N_{\theta}(\varepsilon,\lambda)~:~\varepsilon >0 ~and ~0 <\lambda <1 \}$ forms a local base for some metrizable linear topology on $E$, called the $(\varepsilon,\lambda)$--topology, denoted by $\mathcal{T}_{\varepsilon,\lambda}$. Further, $L^0(\mathcal{F},K)~($ as a special $RN$ module, its $L^0$--norm is just the usual absolute value mapping $|\cdot|)$ is a topological algebra over $K$ when $L^0(\mathcal{F},K)$ is endowed with its $(\varepsilon,\lambda)$--topology, and $(E,\mathcal{T}_{\varepsilon,\lambda})$ is a topological module over the topological algebra $L^0(\mathcal{F},K)$.
\end{proposition}

\begin{example}\label{example2.5}
Let $(B,\|\cdot\|)$ be a Banach space over $K$ and $L^0(\mathcal{F},B)$ the linear space of equivalence classes of $B$--valued strong $\mathcal{F}$--random elements on $(\Omega,\mathcal{F},P)$ under the ordinary scalar multiplication and addition operations on equivalence classes. Then the scalar multiplication on $B$ induces the module multiplication on $L^0(\mathcal{F},B)$ as follows: $\xi x:=$ the equivalence class of $\xi^0(\cdot)\cdot x^0(\cdot)$ for any $(\xi,x) \in L^0(\mathcal{F},K) \times L^0(\mathcal{F},B)$, where $\xi^0$ and $x^0$ are respectively arbitrarily chosen representatives of $\xi$ and $x$. Further the norm $\|\cdot\|$ on $B$ induces the $L^0$--norm on $L^0(\mathcal{F},B)$, still denoted by $\|\cdot\|$ as follows:\\
$\|x\|=$the equivalence class of $\|x^0(\cdot)\|$ for any $x\in L^0(\mathcal{F},B)$, where $x^0$ is as above.\\
Then $(L^0(\mathcal{F},B),\|\cdot\|)$ becomes an $\mathcal{T}_{\varepsilon,\lambda}$--complete $RN$ module over $K$ with base $(\Omega,\mathcal{F},P)$. In particular the $(\varepsilon,\lambda)$--topology $\mathcal{T}_{\varepsilon,\lambda}$ on $L^0(\mathcal{F},B)$ is just the topology of convergence in probability on $L^0(\mathcal{F},B)$.
\end{example}

\par
To give Example \ref{example2.6} below, let us first recall from \cite{B72} the notion of a $w^*$--random element. Let $B'$ be the conjugate space of a Banach space $(B,\|\cdot\|)$ over $K$, a mapping $V: (\Omega,\mathcal{F},P) \rightarrow B'$ is called a $w^*$--$\mathcal{F}$--random element if $V(\cdot)(b): \Omega \rightarrow K$ is a $K$--valued $\mathcal{F}$--random variable on $(\Omega,\mathcal{F},P)$ for any $b\in B$. Furthermore, two $w^*$--$\mathcal{F}$--random elements $V_1$ and $V_2: (\Omega,\mathcal{F},P) \rightarrow B'$  are said to be $w^*$--equivalent if $V_1(\cdot)(b)=V_2(\cdot)(b)$ a.s. for any $b\in B$.

\begin{example}\label{example2.6}
Let $L^0(\mathcal{F},B',w^*)$ be the linear space of $w^*$--equivalence classes of $w^*$--$\mathcal{F}$--random elements from $(\Omega,\mathcal{F},P)$ to $B'$ under the usual scalar multiplication and addition operations on $w^*$--equivalence classes. Similarly to Example \ref{example2.5}, $L^0(\mathcal{F},B',w^*)$ is a left module over the algebra $L^0(\mathcal{F},K)$ under the module multiplication induced from the scalar multiplication on $B'$, defined $\|\cdot\|: L^0(\mathcal{F},B',w^*)\rightarrow L_+^0(\mathcal{F})$ as follows:\\
$\|x\|=$ the equivalence class of $ess.sup(\{|x^0(\cdot)(b)|: b\in B$ and $\|b\|\leq 1\})$ for any $x\in L^0(\mathcal{F},B',w^*)$, where $x^0(\cdot)$ is an arbitrarily chosen representative of $x$. Then $(L^0(\mathcal{F},B',w^*),\|\cdot\|)$ becomes an $\mathcal{T}_{\varepsilon,\lambda}$--complete $RN$ module over $K$ with base $(\Omega,\mathcal{F},P)$. Notice: let $r:\Omega \rightarrow R$ be defined by $r(\omega)=\|x^0(\omega)\|:=sup\{|x^0(\omega)(b)|: b\in B$ and $\|b\|\leq 1\}$, then $|x^0(\cdot)(b)|\leq r$ a.s. for any $b\in B$ and $\|b\|\leq 1$, although $r$ is not necessarily $\mathcal{F}$--measurable, $ess.sup(\{|x^0(\cdot)(b)|: b\in B$ and $\|b\|\leq 1\})$ always belongs to $\mathcal{L}(\mathcal{F})$, so $\|x\|$ is well defined.
\end{example}

\par
The $(\varepsilon,\lambda)$--topology is essentially not locally convex, for example, $L^0(\mathcal{F},K)$ is the simplest $RN$ module, but there does not any nontrivial continuous linear functional on $L^0(\mathcal{F},K)$ when $\mathcal{F}$ does not any atom. The following notion of a random conjugate space is crucial in the subsequent development of $RN$ modules.

\begin{definition}\label{definition2.7}\cite{Guo89,Guo10}
Let $(E,\|\cdot\|)$ be an $RN$ module over $K$ with base $(\Omega,\mathcal{F},P)$. A linear operator $f : E \to L^0(\mathcal{F},K)$ is said to be a.s. bounded if there exists some $\xi \in L^0_+(\mathcal{F})$ such that $|f(x)| \leq \xi \cdot \|x\|$ for any $x \in E$. Denote by $E^*$ the linear space of a.s. bounded linear operators from $E$ to $L^0(\mathcal{F},K)$, a module multiplication $\cdot: L^0(\mathcal{F},K)\times E^*\rightarrow E^*$ is introduced by $(\xi\cdot f)(x)=\xi\cdot (f(x))$ for any $(\xi,f)\in L^0(\mathcal{F},K)\times E^*$, and an $L^0$--norm $\|\cdot\|: E^*\rightarrow L^0_+(\mathcal{F})$ is given by $\|f\|=\bigwedge\{\xi\in L^0_+(\mathcal{F}): |f(x)|\leq \xi\cdot\|x\|$ for any $x\in E\}$ for any $f\in E^*$, then $(E^*,\|\cdot\|)$ becomes an $RN$ module over $K$ with base $(\Omega,\mathcal{F},P)$, called the random conjugate space of $E$.
\end{definition}

\par
Now, it is also well known that $f\in E^*$ iff $f$ is a continuous module homomorphism from $(E,\mathcal{T}_{\varepsilon,\lambda})$ to $(L^0(\mathcal{F},K),\mathcal{T}_{\varepsilon,\lambda})$, at which time $\|f\| = \bigvee \{ |f(x)| : x \in E$ and $\|x\| \leq 1 \}$, see \cite{Guo95,Guo10,Guo13} for details.

\begin{lemma}\cite{Guo96a}\label{lemma2.8}
Let $(\Omega,\mathcal{F},P)$ be a complete probability space and $(B,\|\cdot\|)$ a Banach space over $K$. Define the canonical mapping $J:L^0(\mathcal{F},B',w^*)\rightarrow L^0(\mathcal{F},B)^*$ as follows: for each $y\in L^0(\mathcal{F},B',w^*)$, $J(y): L^0(\mathcal{F},B)\rightarrow L^0(\mathcal{F},K)$ is given by $J(y)(x)=\langle x,y \rangle$ for each $x \in L^0(\mathcal{F},B)$, where $\langle x, y \rangle$ is the equivalence class $\langle x^0(\cdot),y^0(\cdot)\rangle$ defined by $\langle x^0(\omega),y^0(\omega)\rangle=y^0(\omega)(x^0(\omega))$ for each $\omega\in \Omega$, $x^0$ and $y^0$ are a representative of $x\in L^0(\mathcal{F},B)$ and $y\in L^0(\mathcal{F},B',w^*)$, respectively. Then $J$ is an isometric isomorphism from $L^0(\mathcal{F},B',w^*)$ onto $L^0(\mathcal{F},B)^*($in the sense of $RN$ modules$)$.
\end{lemma}

\begin{definition}\cite{GZWW17}\label{definition2.9}
Let $E$ be a topological module over the topological algebra $(L^0(\mathcal{F},K) \\
,\mathcal{T}_{\varepsilon,\lambda})$. A subset $G$ of $E$ is said to be $L^0$--convex if $\xi x+\eta y\in G$ for any $x$ and $y\in G$ and $\xi$ and $\eta\in L^0_+(\mathcal{F})$ such that $\xi+ \eta=1$. Further, an nonempty $L^0$--convex subset $G$ of $E$ is said to be $L^0$--convexly compact$($ or, to have $L^0$--convex compactness$)$ if any family of nonempty closed $L^0$--convex subsets of $G$ has a nonempty intersection whenever the family has the finite intersection property$($ namely each of its finite subfamily has a nonempty intersection$)$.
\end{definition}

\par
Throughout this paper, since every $RN$ module is always endowed with the $(\varepsilon,\lambda)$--topology, all topological terminologies are with respect to the $(\varepsilon,\lambda)$--topology, so we no longer mention the $(\varepsilon,\lambda)$--topology in the sequel of this paper. An $RN$ module $(E,\|\cdot\|)$ over $K$ with base $(\Omega,\mathcal{F},P)$ ia a topological module over the topological algebra $L^0(\mathcal{F},K)$, in particular, we have the following:

\begin{proposition}\cite[Theorem 2.21]{GZWW17}\label{proposition2.10}
Let $(E,\|\cdot\|)$ be a complete $RN$ module over $K$ with base $(\Omega,\mathcal{F},P)$ and $G$ a closed $L^0$--convex subset of $E$. Then $G$ is $L^0$--convex compact if and if for each $f\in E^*$ there exists $g_0\in G$ such that $Re(f(g_0))=\bigvee\{Re(f(g)): g\in G\}$, where $Re(f(g))$ stands for the real part of $f(g)$.
\end{proposition}

Proposition \ref{proposition2.10} is, in fact, a random generalization of the famous James' theorem \cite{J64}, which shows that the notion of $L^0$--convex compactness for a complete $RN$ module plays the same role as weak compactness for a Banach space, Proposition \ref{proposition2.10} directly leads to a positive answer to Problem $(1)$ in Introduction of this paper, namely Theorem \ref{theorem2.11} below:

\begin{theorem}\label{theorem2.11}
Let $(B,\|\cdot\|)$ be a Banach space over $K$ and $V$ a closed convex subset of $B$. Then $L^0(\mathcal{F},V)$ is an $L^0$--convexly compact subset of the complete $RN$ module $L^0(\mathcal{F},B)$ if and only if $V$ is weakly compact. Where $L^0(\mathcal{F},V)$ is the set of equivalence classes of $V$--valued strong $\mathcal{F}$--random elements on $(\Omega,\mathcal{F},P)$, it is clear that $L^0(\mathcal{F},V)$ is a closed $L^0$--convex subset of $L^0(\mathcal{F},B)$.
\end{theorem}

\begin{proof}
Let $\mathcal{F}^P$ be the completion of $\mathcal{F}$ with respect to $P$, since elements in $L^0(\mathcal{F},B)$ are equivalence classes and each $B$--valued $\mathcal{F}^P$--strong random element on $(\Omega,\mathcal{F}^P,P)$ is almost everywhere equal to a $B$--valued $\mathcal{F}$--strong random element on $(\Omega,\mathcal{F},P)$, $L^0(\mathcal{F},B)$ and $L^0(\mathcal{F}^{P},B)$ are identified, we can thus, without loss of generality, assume that $(\Omega,\mathcal{F},P)$ is a complete probability space $($otherwise, we consider $(\Omega,\mathcal{F}^P,P)$ instead of $(\Omega,\mathcal{F},P))$. By Proposition \ref{proposition2.10},  for the part of sufficiency we only need to verify that $Ref$ can attain its maximum on $L^0(\mathcal{F},V)$ for any given $f\in L^0(\mathcal{F},B)^*$.
\par
Since $V$ is weakly compact, $V$ is bounded, then $L^0(\mathcal{F},V)$ is, of course, a.s. bounded, namely $\bigvee \{ \|x\| : x \in L^0(\mathcal{F},V) \} \in L^0_+(\mathcal{F})$, so that $\xi:=\bigvee\{Re(f(x)): x\in L^0(\mathcal{F},V)\}\in L^0_+(\mathcal{F})$. It is easy to see that $\{Re(f(x)),x \in L^0(\mathcal{F},V)\}$ is directed upwards, so there exists a sequence $\{ x_n, n \in N \}$ in $L^0(\mathcal{F},V)$ such that $\{Re(f(x_n)),n\in N\}$ converges a.e. to $\xi$ in a nondecreasing way.
\par
By Lemma \ref{lemma2.8}, there exists $y\in L^0(\mathcal{F},B',w^*)$ such that $f(v)=\langle v,y\rangle$ for each $v\in L^0(\mathcal{F},B)$. Further, let $x^0_n$ be an arbitrarily chosen representative of $x_n$ for each $n\in N$ and $y^0$ an arbitrarily chosen representative of $y$, then it is obvious that $\langle x^0_n(\cdot),y^0(\cdot)\rangle$ is a representative of $\langle x_n,y\rangle$ and $\xi^0:\Omega\rightarrow(-\infty,+\infty)$, defined by $\xi^0(\omega)=sup_{n\geq 1}Re(\langle x^0_n(\omega),y^0(\omega)\rangle)=sup_{n\geq 1}Re(y^0(\omega)(x^0_n(\omega)))$ for each $\omega\in \Omega$, is a representative of $\xi=\bigvee\{Re(f(x)): x\in L^0(\mathcal{F},V)\}$.
\par
If we can prove that there exists a $V$--valued $\mathcal{F}$--strong random element $v^0$ defined on $(\Omega,\mathcal{F},P)$ such that $\xi^0(\omega)=Re(\langle v^0(\omega),y^0(\omega)\rangle)$ for each $\omega \in \Omega$, then $\xi=Re(\langle v,y\rangle)$, we will complete the proof of this lemma, where $v$ is the equivalence class of $v^0$. For this, let $L=\overline{span\{\bigcup_{n\geq 1}x^0_n(\Omega)\}}$, then $L$ is a separable complete subspace of $B$ since each $x^0_n(\Omega)$ is a separable subset of $B$ by the strong measurability of $x^0_n$, further define a multifunction $F:\Omega\rightarrow 2^L$ by $F(\omega)=\overline{conv\{x^0_n(\omega),n\in N\}}$ for each $\omega\in \Omega$, then each $F(\omega)$ is a closed convex subset of $V$, and hence each $F(\omega)$ is a weakly compact convex subset of $V$.
\par
Now, we prove that $F$ is measurable, namely $F^{-1}(G):=\{\omega\in \Omega: F(\omega)\cap G\neq \emptyset\}\in \mathcal{F}$ for each open subset $G$ of $L$. In fact, let $Q$ be the set of rational numbers in $[0,1]$, $Q_1^n=\{(r_1,r_2,\cdots,r_n)\in Q^n:\sum^n_{i=1}r_i=1\}$ and $M_n=\{\sum^n_{i=1}r_ix^0_i:(r_1,r_2,\cdots,r_n)\in Q^n_1\}$ for each $n\in N$, then $M=\bigcup_{n\geq 1}M_n$ is an at most countable set of $V$--valued $\mathcal{F}$--strong random elements, denoted by $\{v_n,n\in N\}$. It is also clear that $F(\omega)=\overline{\{v_n(\omega),n\in N\}}$ for each $\omega\in \Omega$, and thus $F$ is measurable.
\par
Now, define a multifunction $F_1:\Omega\rightarrow 2^L$ by $F_1(\omega)=\{v\in F(\omega): Re(y^0(\omega)(v)) =\xi^0(\omega)\}$ for each $\omega\in \Omega$. Since $y^0(\omega)\in B'$ and $sup\{y^0(\omega)(v):v\in F(\omega)\}=\xi^0(\omega)$ for each $\omega\in \Omega$, then each $F_1(\omega)\neq\emptyset$ since $F(\omega)$ is weakly compact. Further $Gr(F_1)=\{(\omega,v):v\in F_1(\omega)\}=Gr(F)\cap\{(\omega,v)\in\Omega\times L: Re(y^0(\omega)(v))\geq\xi^0(\omega)\}\in \mathcal{F}\bigotimes \mathcal{B}(L)$, so $F_1$ has measurable graph, by Theorem 5.10 of \cite{W77} $F_1$ has a measurable selection $v^0$. Clearly, $v^0$ is a $V$--valued $\mathcal{F}$--strong random element (since $L$ is sparable) such that $Re(y^0(\omega)(v^0(\omega)))=\xi^0(\omega)$ for each $\omega\in \Omega$.
\par
Necessity. Let $L^0(\mathcal{F},V)$ be $L^0$--convexly compact and $f \in B'$, then $f$ induces an element $\hat{f} \in L^0(\mathcal{F},B)^*$ as follows: $\hat{f}(x) =$ the equivalence class of $f(x^0(\cdot))$ for any $x \in L^0(\mathcal{F},B)$, where $x^0(\cdot)$ is an arbitrarily chosen representative of $x$. Then, by Proposition \ref{proposition2.10} there exists some $g_0 \in L^0(\mathcal{F},V)$ such that $Re(\hat{f}(g_0)) = \bigvee \{ Re(\hat{f}(g)) : g \in L^0(\mathcal{F},V) \}$. It is easy to check that $\bigvee \{ Re(\hat{f}(g)) : g \in L^0(\mathcal{F},V) \}$ is the equivalence class of the constant function with its value equal to $\sup \{ Re(f(v)) : v \in V \}$, further letting $g^0$ be an arbitrarily chosen representative of $g_0$ yields $Re(f(g^0(\omega))) = \sup \{ Re(f(v)) : v \in V \}$ for almost all $\omega$ in $\Omega$. Since the $L^0$--convex compactness of $L^0(\mathcal{F},V)$ obviously implies the a.s. boundedness of it, namely $\bigvee \{ \|x\| : x \in L^0(\mathcal{F},V) \} \in L^0_+(\mathcal{F})$, and since it is also easy to check that $\bigvee \{ \|x\| : x \in L^0(\mathcal{F},V) \} $  is the equivalence class of the constant function with its value equal to $\sup \{ \|v\| : v \in V \}$, then $\sup \{ \|v\| : v \in V \} < + \infty$, namely $V$ is a bounded set. Thus $g^0$ is Bochner--integrable, let $v_0= \int_{\Omega} g^0 (\omega) P(d\omega)$, then $Re(f(v_0)) = \int_{\Omega} Re (f(g^0(\omega))) P(d\omega) = \sup \{ Re(f(v)) : v \in V \}$. To sum up, $V$ is weakly compact by the famous James' theorem \cite{J64}.
\end{proof}
\par
Let us recall that a nonempty closed convex subset $V$ of a Banach space $(B,\|\cdot\|)$ is said to have normal structure if for each bounded closed convex subset $H$ of $V$ with $Diam(H) := \sup \{ \|h_1 - h_2 \| : h_1,h_2 \in H \} >0$ there exists at least one point $h_0$ of $H$ such that $\sup \{ \|h_0 -h \| : h \in H \} < Diam(H)$ (such an $h_0$ is called a nondiametral point of $H$). For a nonempty subset $H$ of a complete $RN$ module over $K$ with base $(\Omega,\mathcal{F},P)$, $D(H) := \bigvee \{ \|h_1 - h_2 \| : h_1,h_2 \in H \}$ is called the random diameter of $H$, it is clear that $H$ is a.s. bounded iff $D(H) \in L^0_+(\mathcal{F})$. Similarly, we have the following:
\begin{definition}\label{definition2.12}\cite{GZWG18}
 A nonempty closed $L^0$--convex subset $G$ of a complete $RN$ module over $K$ with base $(\Omega,\mathcal{F},P)$ is said to have random normal structure if for each closed a.s. bounded $L^0$--convex subset $H$ of $G$ with $D(H)>0$ there exists at least one point $h_0 \in H$ such that  $\bigvee\{ \|h_0-h\| : h \in H \} < D(H)$ on $(D(H)>0)$. Such an $h_0$ is called a nondiametral point of $H$.
\end{definition}
\par
In \cite{GZWG18}, we proved that every closed $L^0$--convex subset of a random uniformly convex $RN$ module has random normal structure. Specially, for a closed convex subset $V$ of a uniformly convex Banach space $B$, $L^0(\mathcal{F},V)$ has random normal structure. In this paper, we proved Theorem \ref{theorem2.16} below, which answers Problem (2) in Introduction of this paper. To prove it, we first give Definition \ref{definition2.13}, Lemmas \ref{lemma2.14} and \ref{lemma2.15} below.
\begin{definition}\cite{Guo10}\label{definition2.13}
A nonempty subset $G$ of a left module $E$ over the algebra $L^0(\mathcal{F}, K)$ is said to have the countable concatenation property (or simply, $G$ is stable ) if for each  sequence $\{ g_n: n \in N \}$ of $G$ and each countable partition $\{ A_n: n \in N \}$ of $\Omega$ to $\mathcal{F}$ there exists some $g \in G$ such that ${\tilde I}_{A_n}\cdot g = {\tilde I}_{A_n}\cdot g_n$ for each $n \in N$.
\end{definition}
\par
As pointed out in \cite{Guo10}, for a stable subset $G$ of an $RN$ module over $K$ with base $(\Omega,\mathcal{F},P)$, $g$ as in Definition \ref{definition2.13} must be unique, denoted by $\sum _{n=1}^{\infty} {\tilde I}_{A_n}\cdot g_n$
\begin{lemma}\label{lemma2.14}\cite{GZWG18}
Let $(E,\|\cdot\|)$ and $(E_1,\|\cdot\|_1)$ be two complete $RN$ modules over $K$ with base $(\Omega,\mathcal{F},P)$ $($at this time $E$ and $E_1$ are both stable, see \cite{Guo10,GZWG18}$)$. $G \subset E$ a nonempty subset and $T : G \to E_1$ an $L^0$--Lipschitz mapping $($namely, there exists $\xi \in L^0_+(\mathcal{F})$ such that $\|T(x) - T(y)\|_1 \leq \xi \cdot \|x -y \|$ for all $x$ and $y \in G)$. Then when $G$ is stable, $T$ is stable $($and hence, $T(G)$ is also stable$)$, namely $T(\sum_{n= 1}^\infty \tilde{I}_{A_n} \cdot g_n) = \sum_{n= 1}^\infty \tilde{I}_{A_n} \cdot (T(g_n))$ for each sequence $\{ g_n: n \in N \}$ in $G$ and each countable partition $\{ A_n: n \in N \}$ of $\Omega$ to $\mathcal{F}$.
\end{lemma}
\begin{lemma}\label{lemma2.15}\cite{GZWG18}
Let $G$ be a stable subset of $L^0(\mathcal{F})$ such that $G$ has an upper $($lower$)$ bound $\xi \in L^0(\mathcal{F})$, then for each $\varepsilon \in L^0_{++}(\mathcal{F})$ there exists some $g_{\varepsilon} \in G$ such that $g_{\varepsilon} > \bigvee G - \varepsilon$ on $\Omega$ $($correspondingly, $g_{\varepsilon} < \bigwedge G + \varepsilon$ on $\Omega)$.
\end{lemma}
\begin{theorem}\label{theorem2.16}
Let $(B,\|\cdot\|)$ be a Banach space over $K$ and $V\subset B$ a weakly compact convex subset with normal structure. Then $L^0(\mathcal{F},V)$, as a closed $L^0$--convex subset of $L^0(\mathcal{F},B)$, has random normal structure.
\end{theorem}
\begin{proof}
As in the proof of Theorem \ref{theorem2.11}, we can assume that $(\Omega,\mathcal{F},P)$ is a complete probability measure space.
\par
If $L^0(\mathcal{F},V)$ would not have random normal structure, then there exists some nonempty closed $L^0$--convex and $L^0$--bounded subset $H$ of $L^0(\mathcal{F},V)$ such that $D(H):=\bigvee\{\|h_1-h_2\|:h_1,h_2\in H\}>0$ but there does not exist any point $h$ in $H$ with the property: $\bigvee\{\|h-h_1\|:h_1\in H\}<D(H)$ on $(D(H)>0)$, we prove that this will produce a contradiction. As in the proof of Theorem 3.9 of \cite{GZWG18}, we can assume that $D(H)\in L^0_{++}(\mathcal{F})$, namely $D(H)>0$ on $\Omega$. Let us consider the $L^0$--convex function $f: H\rightarrow L^0(\mathcal{F})$ defined by $f(x)=\bigvee \{\|x-h\|:h\in H\}$ for each $x\in H$, since $L^0(\mathcal{F},V)$ is $L^0$--convexly compact by Theorem \ref{theorem2.11}, $H$ is also $L^0$--convexly compact, further one can easily see that $f$ satisfies all the conditions of Theorem 3.6 of \cite{GZWW17}, so that there exists $h_0\in H$ such that $f(h_0)=$ min $\{f(x):x\in H\}$.
\par
Thus, there exists some $A\in \mathcal{F}$ with $P(A)>0$ such that $\bigvee\{\|h_0-h\|:h\in H\}=D(H)$ on $A$, we can, of course, have that $\bigvee\{\|h_1-h\|:h\in H\}=D(H)$ on $A$ for each $h_1\in H$. Again, for the sake of brevity, we can, without loss of generality, assume that $A=\Omega$. For any given $x_1\in H$, since $\{ x_1 -h : h \in H \}$ is a closed $L^0$--convex subset and hence also stable, and further since $\|\cdot\|$ is $L^0$--Lipschitzian, then by Lemmas \ref{lemma2.14} and \ref{lemma2.15} there exists $x_2 \in H$ such that $\|x_2-x_1\|\geq D(H)-1$, similarly, there exists $x_3\in H$ such that $\|x_3-\frac{x_1+x_2}{2}\|\geq D(H)-\frac{1}{2^2}$. By the induction method there exists a sequence $\{x_n,n\in N\}$ in $H$ such that $\|x_{n+1}-\frac{1}{n}\sum^n_{i=1}x_i\|\geq D(H)-\frac{1}{n^2}$ for any $n\in N$.
\par
For any $x\in conv\{x_k,k\in N\}$, then there exist $l\in N$ and finite nonnegative numbers $\lambda_1,\lambda_2,\cdots,\lambda_l\in [0,1]$ such that $\Sigma^n_{i=1}\lambda_i=1$ and $x=\Sigma^n_{i=1}\lambda_ix_i$, where $conv(\Gamma)$ stands for the convex hull of a subset $\Gamma$ of $L^0(\mathcal{F},B)$. Let $\lambda=$ max $\{\lambda_1,\lambda_2,\cdots,\lambda_l\}$ and $\lambda_{l+1}=\cdots=\lambda_n=0$ for any $n\in N$ such that $n>l$, then
\begin{equation}\nonumber
\aligned
&\|x_{n+1}-\Sigma^l_{i=1}\lambda_ix_i\|\\
&=\|x_{n+1}-\Sigma^n_{i=1}\lambda_ix_i\|\\
&=\|n\lambda x_{n+1}-\lambda\Sigma^n_{i=1}x_i-n\lambda x_{n+1}+\lambda\Sigma^n_{i=1}x_i+\Sigma^n_{i=1}\lambda_ix_{n+1}-\Sigma^n_{i=1}\lambda_ix_i\|\\
&=\|n\lambda(x_{n+1}-\frac{1}{n}\Sigma^n_{i=1}x_i)-\Sigma^n_{i=1}(\lambda-\lambda_i)(x_{n+1}-x_i)\|\\
&\geq n\lambda\|x_{n+1}-\frac{1}{n}\Sigma^n_{i=1}x_i\|-\Sigma^n_{i=1}(\lambda-\lambda_i)\|x_{n+1}-x_i\|\\
&\geq n\lambda(D(H)-\frac{1}{n^2})-n\lambda D(H)+D(H)\\
&=D(H)-\frac{\lambda}{n}\\
&\geq D(H)-\frac{1}{n}.
\endaligned
\end{equation}
\par
So, for any $x\in conv\{x_k,k\in N\}$, $\{\|x_n-x\|,n\in N\}$ converges a.e. to $D(H)$, which also means that $\bigvee\{\|x-y\|:y\in conv\{x_k,k\in N\}\}\geq \bigvee\{\|x-x_n\|:n\in N\}=D(H)$ for any $x\in conv\{x_k,k\in N\}$. Further, it is also clear that $\bigvee\{\|x-y\|:y\in conv\{x_k,k\in N\}\}=\bigvee\{\|x_i-x_j\|:i,j\in N\}=D(H)$ for any $x\in conv\{x_k, k\in N\}$.
\par
Now, arbitrarily choose a representative $x^0_n$ of $x_n$ for each $n\in N$, let $\{v_n,n\in N\}$ be constructed as in the proof of Theorem \ref{theorem2.11}, then $sup\{\|v_n(\omega)-v_m(\omega)\|:m\in N\}=sup\{\|x^0_i(\omega)-x^0_j(\omega)\|: i,j\in N\}$ for each $n\in N$ and for almost all $\omega\in \Omega$. We can, without loss of generality, assume that $sup\{\|v_n(\omega)-v_m(\omega)\|:m\in N\}=sup\{\|x^0_i(\omega)-x^0_j(\omega)\|: i,j\in N\}$ for each $n\in N$ and each $\omega\in \Omega$. Let $L=\overline{span\{\bigcup^\infty_{n=1}x^0_n(\Omega)\}}$ and define $F:\Omega \rightarrow 2^L$ by $F(\omega)=\overline{conv\{x^0_n(\omega),n\in N\}}$ for each $\omega\in \Omega$, then $L$ is a complete separable subspace of $B$ and each $F(\omega)$ is a closed convex subset of $V$. Since $F(\omega)=\overline{\{v_n(\omega),n\in N\}}$ for each $\omega\in \Omega$, it is obvious that $sup\{\|u-v\|: v\in F(\omega)\}=sup\{\|x^0_i(\omega)-x^0_j(\omega)\|:i,j\in N\}=sup\{\|v_n(\omega)-v_m(\omega)\|: n,m\in N\}=Diam(F(\omega))$ for each $\omega \in \Omega$ and for each $u \in F(\omega)$. This means that each $F(\omega)$ does not possess any nondiametral point, which contradicts the fact that $V$ has normal structure.
\end{proof}
\begin{remark}\label{remark2.17}
The idea of constructing the sequence $\{x_n,n\in N\}$ in the proof of Theorem \ref{theorem2.16} is motivated from Brodski\v{i} and Mil'man's work \cite{BM68}, normal structure of Banach spaces was deeply studied in \cite{BM68,L84,ST90}, in particular Smith and Turett proved in \cite{ST90} that a Banach space $X$ has normal structure iff the Lebesgue--Bochner function space $L^p(\mu,X)$ has normal structure for any given $p$ such that $1<p<+\infty$, although our proof of Theorem \ref{theorem2.16} depends on the weak compactness of $V$ (the weak compactness is used to guarantee that there exists some measurable set $A$ with positive measure such that $\bigvee\{\|h_1-h\|: h\in H\}=D(H)$ on $A$ for each $h_1\in H$), we indeed wonder whether Theorem \ref{theorem2.16} is true or not when $V$ only satisfies the single condition that $V$ has normal structure, namely, is the weak compactness of $V$ superfluous in Theorem \ref{theorem2.16} ?
\end{remark}
\section{A general random fixed point theorem for a nonexpansive strong random operator}
The main result of this section is Theorem \ref{theorem3.2}, which provides a best solution to Bharucha--Reid's problem, and is based on a fixed point theorem recently obtained in \cite{GZWG18} as follows:
\begin{proposition}\label{proposition3.1}\cite{GZWG18}
Let $(E,\|\cdot\|)$ be a complete $RN$ module over $K$ with base $(\Omega, \mathcal{F}, P)$ and $G \subset E$ an $L^0$--convexly compact $L^0$--convex subset with random normal structure. Then every nonexpansive mapping $T$ form $G$ to $G$ $($namely $\|T(u)-T(v)\|\leq \|u-v\|$ for all $u,v \in G)$ has a fixed point in $G$.
\end{proposition}
\begin{theorem}\label{theorem3.2}
Let $(B,\|\cdot\|)$ be a Banach space over $K$ and $V$ a weakly compact convex subset of $B$ such that $V$ has normal structure, then every strong random nonexpansive operator $T : (\Omega,\mathcal{F},\mu) \times V \to V$ has a strong $\mathcal{F}$--random element $x^0(\cdot) : \Omega \to V$ such that $T(\omega,x^0(\omega)) = x^0(\omega)$ for almost all $\omega \in \Omega$.
\end{theorem}
\begin{proof}
By Theorems \ref{theorem2.11} and \ref{theorem2.16}, $L^0(\mathcal{F},V)$ has both $L^0$--convex compactness and random normal structure. Further, $T$ induces a nonexpansive mapping $\hat{T}: L^0(\mathcal{F},V)\rightarrow L^0(\mathcal{F},V)$ in a natural manner: for each given $x\in L^0(\mathcal{F},V)$, arbitrarily chose a representative $x^0$ of $x$, then define $\hat{T}(x)$ as the equivalence class of $T(\cdot,x^0(\cdot))$, then $T(\cdot,x^0(\cdot))$ is a $V$--valued strong $\mathcal{F}$--random element since $T$ is nonexpansive (and thus also continuous) and a strong random operator, so that $\hat{T}$ is well defined. Applying Proposition \ref{proposition3.1} to $\hat{T}$ and $L^0(\mathcal{F},V)$ produces some $x\in L^0(\mathcal{F},V)$ such that $\hat{T}(x)=x$, then an arbitrarily chosen representative $x^0$ of $x$ must satisfy $T(\omega,x^0(\omega))=x^0(\omega)$ for almost all $\omega\in\Omega$.
\end{proof}
\begin{remark}\label{remark3.3}

Up to now, Theorem \ref{theorem3.2} also provides a most general partial answer to the question posed by Xu in Remark 1 of \cite{Xu90}, namely in the general case of Theorem \ref{theorem3.2} the assumption that $\mathcal{F}$ is closed under the Suslin operation is indeed superfluous.

\end{remark}
\par
Through Corollaries \ref{corollary3.4} and \ref{corollary3.5} below we illustrate that the fixed point theorems provided by this paper can generalize and improve the random fixed point theorems currently available for nonexpansive random self--mappings, Corollary \ref{corollary3.4} generalizes and improves Lemma 1 of \cite{L88} in that a strong random operator is employed to replace the assumption on the separability in \cite{L88} and we also remove the assumption in \cite{L88} that $\sum$ is closed under the Suslin operation. Similarly, Corollary \ref{corollary3.5} generalizes and improves Theorems $3^{\prime}$ and $6^{\prime}$ of \cite{L88}.
\begin{corollary}\label{corollary3.4}
Let $S$ be a nonempty closed convex subset of a uniformly convex Banach space $X$ and $f: \Omega\times S\rightarrow S$ a nonexpansive strong random operator such that $f(\omega,S)$ is bounded for any $\omega\in \Omega$. Then $f$ has a strongly measurable random fixed point.
\end{corollary}

\begin{proof}
Let $r:\Omega\rightarrow [0,+\infty)$ be the real--valued function defined by $r(\omega)=sup\{\|f(\omega,s)\|: s\in S\}$, then for each $S$--valued $\mathcal{F}$--strong random element $x^0:\Omega\rightarrow S$, $f(\cdot,x^0(\cdot))$ is an $S$--valued strong $\mathcal{F}$--random element since $f$ is a nonexpansive strong random operator, further, $\|f(\omega,x^0(\omega))\|\leq r(\omega)$ for each $\omega\in \Omega$, denote $\xi^0=esssup\{\|f(\cdot,x^0(\cdot))\|: x^0$ is an $S$--valued and strong $\mathcal{F}$--random element $\}$, then $\xi^0$ is a nonnegative real--valued $\mathcal{F}$--random variable on $(\Omega,\mathcal{F},P)$ and $\xi^0(\omega)\leq r(\omega)$ for almost all $\omega\in\Omega$.
\par
Define $\hat{f}: L^0(\mathcal{F},S)\rightarrow L^0(\mathcal{F},S)$ by $\hat{f}(x)=$ the equivalence class of $f(\cdot,x^0(\cdot))$ for each $x\in L^0(\mathcal{F},S)$ with $x^0$ as a representative of $x$, then $\|\hat{f}(x)\|\leq\xi$ for each $x\in L^0(\mathcal{F},S)$, where $\xi$ stands for the equivalence class of $\xi^0$, namely $\hat{f}(L^0(\mathcal{F},S))$ is a.s. bounded. By Corollary 3.12 and Remark 3.13 of \cite{GZWG18}, $\hat{f}$ has a fixed point $x$ in $L^0(\mathcal{F},S)$, then a representative $x^0$ of $x$ must satisfy $f(\omega,x^0(\omega))=x^0(\omega)$ for almost all $\omega\in \Omega$.

\end{proof}

\begin{corollary}\label{corollary3.5}
Let $S$ be a nonempty closed convex subset of a Hilbert space $X$ and $f: \Omega\times S\rightarrow X$ a nonexpansive strong random operator such that $f(\omega,S)$ is bounded for any $\omega\in \Omega$. Then there exists an $S$--valued strong $\mathcal{F}$--random element $\varphi: \Omega\rightarrow S$ such that $\|\varphi(\omega)-f(\omega,\varphi(\omega))\|=d(f(\omega,\varphi(\omega)),S)$ for almost all $\omega\in\Omega$. Further, $\varphi$ is also a random fixed point of $f$ if $f$ satisfies, in addition, one of the following two conditions:
\begin{enumerate}[(i)]
\item For each $\omega \in \Omega$ and each $x \in S$ with $x \neq f(\omega,x)$, there exists $y$, depending on $\omega$ and $x$, in $I_S(x) = \{ x+c(z-x) : z \in S$ and $c \geq 0\}$ such that $\|y-f(\omega,x)\| < \|x-f(\omega,x)\|$.
\item $f$ is weakly inward, namely, $f(\omega,x) \in \overline{I_S(x)}$ for each $\omega \in \Omega$ and each $x \in S$.
\end{enumerate}
\end{corollary}
\begin{proof}
Let $p:X\rightarrow S$ be the usual proximity mapping, then $p\circ f : \Omega \times S \to S$ is a nonexpansive strong random operator such that $p\circ f(\omega,S)$ is bounded for any $\omega\in\Omega$. By Corollary \ref{corollary3.4}, $p\circ f$ has an $S$--valued $\mathcal{F}$--strongly measurable random fixed point $\varphi$, which satisfies $\|\varphi(\omega) - f(\omega,\varphi(\omega))\| = d(f(\omega,\varphi(\omega)),S)$ for almost all $\omega \in \Omega$. The remaining part of this proof may proceed as in the proof of Theorem 4 of \cite{L88}.
\end{proof}

\noindent{\bf Acknowledgments:}

This work was supported by National Natural Science Foundation of China (Grant No. 11571369).

\end{document}